\documentclass[12pt, reqno,]{amsart}
\usepackage{amsmath}
\numberwithin{equation}{section}
\usepackage{hyperref}
\hypersetup{nesting=true,debug=true,naturalnames=true}
\usepackage{graphicx,amssymb,upref}
\usepackage{enumerate}
\usepackage[all]{xy}
\usepackage{color}
\usepackage{mathrsfs}
\usepackage{amsfonts}
\usepackage{caption}
\usepackage{paralist}
\usepackage{color}
\usepackage{float}
\usepackage{tikz,pgf}
\usetikzlibrary{calc}
\usetikzlibrary{intersections,decorations.markings}

\theoremstyle{plain}
\newtheorem{theorem}{\bf Theorem}[section]

\newtheorem{lemma}[theorem]{\bf Lemma}
\newtheorem{corollary}[theorem]{\bf Corollary}
\newtheorem{conjecture}[theorem]{\bf Conjecture}
\theoremstyle{definition}

\newtheorem{definition}[theorem]{\bf Definition}
\newtheorem{remark}[theorem]{\bf Remark}

\newtheorem{proposition}[theorem]{\bf Proposition}
\newcommand{\mm}[1]{\mathrm{#1}}
\newcommand{\mb}[1]{\mathbb{#1}}
\newcommand{\mc}[1]{\mathcal{#1}}
\makeatletter
\@namedef{subjclassname@2020}{\textup{2020} Mathematics Subject Classification}
\makeatother

\begin{document}
\title[PL cobordism and classification of PL manifolds]{PL cobordism and classification of PL manifolds}

\author[Wen Shen]{Wen Shen}
\address{Department of Mathematics, Capital Normal University,
Beijing, P.R.China }
\email{shenwen121212@163.com}

\begin{abstract}
In this paper, we present more conclusions for $C^\ast(\mathrm{PL})$ and oriented PL cobordism. Furthermore, we compute the string PL bordism groups of certain dimensions. Finally, we apply these results to classify certain PL manifolds of dimensions $10$ and $13$. 	
\end{abstract}

\subjclass[2020]{Primary 57Q20, 57Q25, 55T15}

\maketitle
\tableofcontents

\section{Introduction}
Since Kervaire \cite{Kerva1960} gave an invariant of $(4k + 2)$-manifolds for $k \ne  0, 1, 3$, the invariant, called Kervaire invariant, has been of interest.  In particular, the existence of framed manifolds of Kervaire invariant one is one of the oldest issues in differential and algebraic topology, and has been solved completely by \cite{BJM} \cite{BrowderW1969} \cite{HHR2016} \cite{Jones} \cite{MaTa} \cite{May} \cite{Milg}, and the recent work of Lin-Wang-Xu.


From \cite{Kerva1960}, the Kervaire invariant can be defined for any $4$-connected closed manifold of dimension $10$, and we see that the Kervaire invariant of any smooth $4$-connected closed manifold of dimension $10$ is $0$. Naturally, we will ask whether a $4$-connected closed manifold of dimension $10$ admits a smooth structure if its Kervaire invariant equals to $0$.
 Kervaire reduced this question to the following conjecture.
\begin{conjecture}
	The $5$-th Betti number $B_5(M)$ and Kervaire invariant $\Phi(M)\in \mb{Z}_2$ are a complete set of invariants of the homotopy type of the  $4$-connected closed manifold $M$ of dimension $10$.
\end{conjecture}

In this paper, we prove the conjecture of Kervaire as follows. 
\begin{theorem}\label{classify}
	The $5$-th Betti number $B_5(M)$ and Kervaire invariant $\Phi(M)\in \mb{Z}_2$ are a complete set of invariants of the homeomorphism type of the  $4$-connected closed manifold $M$ of dimension $10$.
\end{theorem}

Indeed, any $4$-connected closed manifold of dimension $10$ admits a unique piecewise linear (PL) structure. So these invariants in  Theorem \ref{classify} also determine the PL homeomorphism types of the $4$-connected closed manifolds of dimension $10$. 

Consider all $4$-connected closed manifolds of dimension $10$ with the fixed $B_5=2s>0$.
 By observing that $\Phi$ is additive with respect to the connected sum of manifolds, one of the homeomorphism types is represented by the connected sum of $s$ copies of $ S^5\times S^5$, the other type is the connected sum of $M_0$ and $s-1$ copies of $ S^5\times S^5$ where the manifold $M_0$ was constructed by Kervaire in \cite{Kerva1960}.
Then, we have
\begin{corollary}
	 The Kervaire invariant of a $4$-connected closed manifold of dimension $10$ is $0$ iff it admits a smooth structure.
\end{corollary}  


Classification of the manifolds with a given cohomology ring (up to diffeomorphism, homeomorphism, or homotopy) is one of the central problems in geometric topology. The Poincar$\acute{e}$ theorem is a special case. For a more complicated cohomology ring, there also exist many investigations, such as the classifications of certain 5-manifolds \cite{Barden, KreckSu}, simply connected $6$-manifolds \cite{Jupp, Wall1966}, the $7$-manifolds modeled on Aloff-Wallach spaces \cite{KS1, KS, Kruggel}, the Bazaikin spaces \cite{FangShen}. 
Different from the above classifications that focus on smooth manifolds, Theorem \ref{classify} and the following classification are for certain PL manifolds.  

Let $\mathcal M$ denote certain PL $13$-manifolds, which have the same cohomology ring as $\mm{CP}^3\times S^7$. Partial these manifolds have been considered in the smooth category \cite{ShenW}. Recall a definition in \cite{ShenW}
\begin{definition}\label{reslift}
	Let $ x:\mathcal M\to \mm{CP}^\infty $ denote a generator of $ H^2( \mc{M};\mb{Z})$. We call that $\mathcal M$ admits a {\it restriction lift} if there exists a lift $f:\mc{M}\to \mm{CP}^4$ such that the following diagram is homotopy commutative
	\[
\xymatrix@C=0.8cm{
&&\mm{CP}^4\ar[d]^-{}\\
\mathcal M \ar[rr]^-{x}\ar@{.>}[rru]^-{f}& &  \mm{CP}^\infty
}
\] 
\end{definition}

Now, we are ready to present the classifying  theorem.
\begin{theorem}\label{classify13}
	Assume that $\mathcal M$, $\mathcal M^\prime$ admit {\it restriction lifts}. Then 
		\item[(1)] $\mathcal M$ is $\mm{PL}$-homeomorphic to $\mathcal M^\prime$ if and only if $p_1(\mathcal M)=p_1(\mathcal M^\prime)$.
		\item[(2)] $\mathcal M$ is homotopy equivalent to $\mathcal M^\prime$ if and only if $p_1(\mathcal M)=p_1(\mathcal M^\prime)\mod 24$.
\end{theorem}

\begin{remark}\label{PonStieclass}
	Although $\mathcal M$, $\mathcal M^\prime$ are $\mm{PL}$ manifolds, we can also define the same Stiefel-Whitney classes $w_i$ for $i< 7$, and Pontryagin class $p_1$ since $H_i(\mm{BPL},\mm{BO};\mb{Z})=0$ for $i\le 7$. 
\end{remark}

\begin{remark}\label{bundle}
	For any $s\in \mb{Z}$, there exists an $8$-dimensional vector bundle $\eta$ over $\mm{CP}^3$ so that $p_1(\eta)=sx^2$. The total space $\overline{\mc{M}}$ of the sphere bundle of $\eta$ has the first Pontrjagin class $(4+s)x^2$. By the (1) of Theorem \ref{classify13}, any $\mc{M}$ with {\it restriction lift} is PL-homeomorphic to such the total space $\overline{\mc{M}}$.
\end{remark}

The $h$-cobordism theorem is a general method to classify simply connected manifolds.
We mainly use the PL cobordism and surgery to construct the $h$-cobordism, thus proving Theorem \ref{classify} and \ref{classify13}. Next, we introduce some conclusions of PL cobordism.

Throughout the rest of the paper, let $H^i(-)$ (or $H_i(-)$) denote the $\mb{Z}_2$-cohomology (or homology) group $H^i(-;\mathbb{Z}_2)$ (or $H_i(-;\mathbb{Z}_2)$).

The methods developed by Thom to study the unoriented differentiable cobordism ring $\Omega_\ast^{\mm{O}}$ have been applied successfully to various classes of smooth manifolds. Through the theory of microbundles developed by Milnor \cite{Milnor1964}, it is possible to apply Thom's ideas to study PL manifolds. 

Given a vector (or PL) bundle $\pi:E\to X$, one can define a right $\mathscr{A}$ action on $H^\ast(X)$ (see \cite{BrownPeter}) via
\begin{equation}
	x \cdot \alpha=\Psi^{-1}(\chi(\alpha) \Psi(x)) \label{rightaction}
\end{equation} 
for $x\in H^\ast(X)$, $\alpha\in \mathscr{A}$ where $\chi$ is the conjugation of $\mathscr{A}$, $T(E)$ is the Thom space of $E$, and $\Psi$ is the Thom isomorphism $\Psi: H^\ast(X) \to H^\ast(T(E))$. 

Analogous to vector bundles, there exists corresponding Thom isomorphism and classifying space for PL bundles \cite{Rudyak}.
Let BPL be the universal classifying space for stable PL bundles. $H^\ast(\mm{BPL})$ is a Hopf algebra, as to $H^\ast(\mm{BO})$. From \cite{BLP1966}, there is a Hopf algebra $C^\ast(\mm{PL})$ over the mod $2$ Steenrod algebra $\mathscr A$ such that 
\begin{equation}
	H^\ast(\mm{BPL})\cong H^\ast(\mm{BO})\otimes C^\ast(\mm{PL}) \label{PLtoBO}
\end{equation} 
 as Hopf algebras over $\mathscr{A}$, and  $\Omega_\ast^{\mm{PL}}\cong \Omega_\ast^{\mm{O}}\otimes C^\ast(\mm{PL})$ as algebras.
 In particular, the isomorphism (\ref{PLtoBO}) is an isomorphism of right $\mathscr A$-modules, and the right $\mathscr{A}$ action on $H^\ast(\mm{BO})\otimes C^\ast(\mm{PL})$ is given by the formula
 \begin{equation}
 (b \otimes c)\cdot  \alpha =\Sigma b\cdot  \alpha^\prime \otimes \chi(\alpha^{\prime \prime})c \label{rightmodule}
 \end{equation}
  where $\Delta(\alpha) =\Sigma \alpha^{\prime }\otimes \alpha^{\prime \prime} $ is the diagonal map in $\mathscr{A}$. 
 
  It is well known that the homotopy fiber $\mathrm{PL/O}$ of the natural map
$\mathscr{D} :\mathrm{BO}\to \mathrm{BPL}$ 
 is 6-connected. Then we 
take the $(n-1)$-connected covers of BPL for $n=2,4,8$, analogous to BO. 
Moreover, we have the natural map $\mathscr{D}\langle n\rangle :\mathrm{BO}\langle n\rangle\to \mathrm{BPL}\langle n\rangle$.
\begin{theorem}\label{PLtoBO248}
Let the right $\mathscr{A}$ action on $H^\ast(\mm{BO}\langle n\rangle)\otimes  C^\ast(\mm{PL})$ be given by the formula $($\ref{rightmodule}$)$. Then 
		$H^\ast(\mm{BPL}\langle n\rangle)\cong H^\ast(\mm{BO}\langle n\rangle)\otimes  C^\ast(\mm{PL})$ as right $\mathscr{A}$-modules for $n=2,4,8$. 
\end{theorem}
The cases for $n=2,4$ have been proved in \cite{BLP1966}. We only prove the case $n=8$ in this paper.

For the Hopf algebra $C^\ast(\mm{PL})$, we have
\begin{theorem}\label{CastPL}
	$(1)$ $C^i(\mm{PL})=0$ for $1\le i\le 7$.
	
	$(2)$ $C^8(\mm{PL})=\mb{Z}_2$, $C^9(\mm{PL})=\mb{Z}_2^2$.

$(3)$ $C^{10}(\mm{PL})=\mb{Z}^3_2$, $C^{11}(\mm{PL})=\mb{Z}_2^2$.

$(4)$ $C^i(\mm{PL})\ne 0$ for $i\ge 12$.

The notation $\mb{Z}^j_2$ denotes the direct sum of $j$ groups with order $2$.
\end{theorem}

The (1) and (2) in Theorem \ref{CastPL} are known (see \cite{Wall1964}), and the (3) and (4) are proved in this paper. 

Let $\Omega_n^{\mm{SPL}}$ be the oriented cobordism group of PL manifolds. In \cite{Will1966}, Williamson presented partial structures of the groups $\Omega_n^{\mm{SPL}}$ for $0\le n\le 18$. For example, $\Omega_{10}^{\mm{SPL}}=\mb{Z}_2\oplus \mm{2gp}$, $\Omega_{11}^{\mm{SPL}}=\mb{Z}_2\oplus \mb{Z}_3\oplus \mm{2gp}$ where $\mm{2gp}$ denotes a group whose order is a power of $2$. In this paper, we see more precise conclusions for the dimensions $n=10,11$.
\begin{proposition}
	$\Omega_{10}^{\mm{SPL}}=\mb{Z}_2\oplus \mb{Z}_2^2$, $\Omega_{11}^{\mm{SPL}}=\mb{Z}_2\oplus \mb{Z}_3$.
\end{proposition} 
 
 As to smooth string cobordism $\Omega_\ast^{\mm{O}\langle 8\rangle}$, one can define the string cobordism $\Omega_\ast^{\mm{PL}\langle 8\rangle}$ in the PL category. For certain dimensions, we have
 \begin{proposition}\label{piMPL8}
 $\Omega_{n}^{\mm{O}\langle 8\rangle}=\Omega_{n}^{\mm{PL}\langle 8\rangle}$ for $ n\le 7$, $\Omega_8^{\mm{PL}\langle 8\rangle}=\mb{Z}\oplus \mb{Z}_4$,
  
 $\Omega_9^{\mm{PL}\langle 8\rangle}=0$, $\Omega_{10}^{\mm{PL}\langle 8\rangle}=\mb{Z}_2$,
 $\Omega_{11}^{\mm{PL}\langle 8\rangle}=0$, $\Omega_{13}^{\mm{PL}\langle 8\rangle}=0$.
 \end{proposition}  
By considering the map $\Omega_n^{\mm{O}\langle 8\rangle}\to \Omega_n^{\mm{PL}\langle 8\rangle}$, we see
\begin{corollary}
	There exists an $(n-1)$-connected closed $\mm{PL}$ manifold of dimension $2n$ for $n=4,5$, which does not admit any smooth structure.
\end{corollary}
The above corollary implies that the generators with finite order in $\Omega_n^{\mm{PL}\langle 8\rangle}$ for $n=8,10$ can not be represented by any smooth manifolds as in \cite{Wall1962}. In particular, the nontrivial element of $\Omega_{10}^{\mm{PL}\langle 8\rangle}$ can be represented by the manifold $M_0$. Moreover, the corollary gives an answer for the Milnor's question in the theorem $4$ of \cite{Milnor1956}. 

Note that every $4$-connected closed manifold of dimension $10$ admits a unique $\mm{BPL}\langle 8\rangle$-structure, i.e. the classifying map of its PL normal bundle has a lift to $\mm{BPL}\langle 8\rangle$, unique up to homotopy. By $\Omega_{10}^{\mm{PL}\langle 8\rangle}=\mb{Z}_2$, all $4$-connected closed manifolds of dimension $10$ are divided into two $\mm{PL}\langle 8\rangle$ bordism classes, which are distinguished by the Kervaire invariant. Then, Theorem \ref{classify} follows by the classifying theorem \ref{PLhome}.

The plan of this paper is then as follows. In Section 2, we do the preparatory work
for $C^\ast(\mm{PL})$ and $\Omega_\ast^{\mm{SPL}}$; in Section 3 we compute $\Omega_\ast^{\mm{PL}\langle 8\rangle}$ in certain dimensions; in Section 4 we apply the PL cobordism to prove the conjecture of Kervaire;  another application in Section 5 is to classify certain PL $13$-manifolds. 

\section{$C^\ast(\mm{PL})$ and $\Omega_\ast^{\mm{SPL}}$ cobordism}\label{CPL}
We first recall the isomorphism of Hopf algebras \cite{BMM1973}  
\begin{equation}
	H_\ast(\mm{BSPL}) \cong H_\ast(\mm{BSO})\otimes V^{(1)}(K[3]\oplus R\langle 4\rangle)\otimes \Gamma(T)\otimes H_\ast(\mathscr{K}) \label{2localiso}
\end{equation}
where $\mathscr{K}$ is the product of Eilenberg-MacLane spaces $$\Pi_{n\ne 2^i} \mm{K}(\mb{Z}_2, 4n-2) \times \Pi_{n>1} \mm{K}(\mb{Z},4n).$$
Other notations will be explained briefly as follows.
 
$R[k]$ is a coalgebra over $\mathscr A$, its dual $R[k]^\ast$ is a polynomial algebra,
	 $$R[k]^\ast=P[\xi_{1k},\cdots,\xi_{kk}]\quad \mm{deg}(\xi_{ik})=2^{k-i}(2^i-1)$$
	 and is an unstable algebra over $\mathscr A$. Let
$R\langle n\rangle=\oplus_{k\ge n}R[k]$.

Let $\mm{PH}^\ast(\mm{K}(\mb{Z}, 4))$ be the primitive elements of $H^\ast(\mm{K}(\mb{Z}, 4))$, and let 	  
$\alpha: \mm{PH}^\ast(\mm{K}(\mb{Z}, 4)) \to  R[3]^\ast$ be the $\mathscr A$-module homomorphism
with $\alpha(\ell_4) = \xi_{13}$. Indeed, the homomorphism $\alpha$ is a monomorphism \cite{BMM1973}. Let $K[3]$ be the kernel of the dual epimorphism
$\alpha^\ast: R[3] \to  Q(H_\ast(\mm{K}(\mb{Z}, 4)))$ where $Q(-)$ is the vector space of indecomposable elements.

For a graded module $N$, let $sN$ be $N$ with degrees shifted up by $1$. We define $V^{(1)}(N)=T(s\bar{N})/\mathscr{F}$ where $T(s\bar{N})$ is the tensor algebra on $s\bar{N}$, $\bar N\subset N$ the subspace of elements of positive dimension, and $\mathscr F$ is a certain ideal.

There is a monomorphism $i_1:E(R[2])\to H_\ast(\mathscr{K}_1)$ where $E(R[2])$ is the exterior algebra, $\mathscr{K}_1= \Pi_{n\ge 2} \mm{K}(\mb{Z}_2,2^n-2).$ Let $\Gamma(T)=H_\ast(\mathscr{K}_1)//i_1$. 

One can see more details about the above notations in \cite{BMM1973}. Here we pay more attention to the structure of $\mb{Z}_2$-vector space. 

Let $\mathscr{V}=V^{(1)}(K[3]\oplus R\langle 4\rangle)\otimes \Gamma(T)$. 
Since the isomorphism (\ref{PLtoBO}), we see that $C^\ast(\mm{PL})$ is isomorphic to $\mathscr{V}\otimes H_\ast(\mathscr{K})$ as vector spaces.

\begin{lemma}\label{VdePro}
Let $\mathscr{V}_i$ be the set of all elements of degree $i$. $\mathscr{V}_8=0$, $\mathscr{V}_9=\mb{Z}_2^2$, $\mathscr{V}_{10}=\mb{Z}_2$, $\mathscr{V}_{11}=0$. Furthermore, one of generators of $\mathscr{V}_9$ is provided by $\Gamma(T)$, $\mathscr{V}_{10}\subset\Gamma(T)$.
\end{lemma} 
\begin{proof}
	Note that $C^8(\mm{PL})=\mb{Z}_2$ and $C^9(\mm{PL})=\mb{Z}^2_2$ \cite{Wall1964}. By $H_{8}(\mathscr{K})=\mb{Z}_2$ and $H_{9}(\mathscr{K})=0$, we have $\mathscr{V}_8=0$ and $\mathscr{V}_9=\mb{Z}_2^2$. Since $H_9(\mathscr{K}_1)=\mb{Z}_2^8$ and $E(R[2])_{9}=\mb{Z}_2^7$, we have $\Gamma(T)_9=\mb{Z}_2$.
		 
  Since that $K[3]\oplus R\langle 4\rangle$ does not contain any element of degree $9$, $V^{(1)}({K[3]\oplus R\langle 4\rangle})$  does not contain any element of degree $10$. Since  $H_{10}(\mathscr{K}_1)=\mb{Z}_2^{9}$ and  $E(R[2])=\mb{Z}_2^8$, $\mathscr{V}_{10}=\Gamma(T)_{10}=\mb{Z}_2$.  
    	 
 Note that ${ R[ 3]}$ only has an element of degree $10$, which is the dual of $\xi_{13}\xi_{23}$ in $R[3]^\ast$. Note that $\alpha(\mm{Sq}^4\mm{Sq}^2\ell_4)=\xi_{13}\xi_{23}$. From the definition, $K[3]$ does not have any element of degree $10$. Since that the smallest degree of the elements of $R\langle 4\rangle $ is $8$, and the smallest degree of $R[ 3] $ is $4$, we see that 
  $V^{(1)}({K[3]\oplus R\langle 4\rangle})$ does not have any element of degree $11$. By $H_{11}(\mathscr{K}_1)=\mb{Z}_2^{11}$ and $E(R[2])_{11}=\mb{Z}_2^{11}$, $\mathscr{V}_{11}=\Gamma(T)_{11}=0$. 
\end{proof}

Recall the work of Sullivan, i.e. there is a map $$G/\mm{PL}\to \mm{K}(\mb{Z}_2,2) \times_{\delta\mm{Sq}^2} \mm{K}(\mb{Z}, 4) \times \Pi_{n\ge 1}(\mm{K}(\mb{Z}_2, 4n +2) \times \mm{K}(\mb{Z}, 4n +4) )$$
inducing an isomorphism in $\mb{Z}_2$ cohomology. Let $k_i\in H^i(\mm{G/PL})$ correspond to the fundamental class $\ell_i\in H^i(\mm{K}(\mb{Z}_2,i))$ or $H^i(\mm{K}(\mb{Z},i))$. From \cite{BMM1971} and \cite{BMM1973}, we have the following lemma.
\begin{lemma}\label{K2primehomology}
	 $\Gamma(T)\otimes H_\ast(\mathscr{K})$ as the subspace of $H_\ast(\mm{G/PL})$ is mapped injectively into $H_\ast(\mm{BSPL})$ through the natural homomorphism
	 $$H_\ast(\mm{G/PL})\to H_\ast(\mm{BSPL}).$$
	 Conversely, for any $i\ne 2^j-1$, there exists an element $x\in C^{2i}(\mm{PL})$ such that its image is $k_{2i}+ \cdots$ under the natural homomorphism $$ H^\ast(\mm{BSPL})\to H^\ast(\mm{G/PL}).$$ 
\end{lemma} 
\begin{corollary}
	$C^n(\mm{PL})\ne 0$ for $n\ge 12$.
\end{corollary}
\begin{proof}
	Note that $4i\ne 2^{j+1}-2$. Let $i\ge 2$. By Lemma \ref{K2primehomology}, there exists an element $x_{4i}\in C^{4i}(\mm{PL})$ such that its image is $k_{2i}+ \cdots$. Applying the Steenrod operations, we have $\mm{Sq}^kx_{4i}\ne 0$, thus $C^{4i+k}(\mm{PL})\ne 0$ for any $i\ge 2$, $2\le k\le 5$. 
\end{proof}

From \cite{Wall1964}, $C^8(\mm{PL})=\mb{Z}_2$, $C^9(\mm{PL})=\mb{Z}_2^2$.
Let $p_8$ be the generator of $C^8(\mm{PL})$. By Lemma \ref{K2primehomology}, we have
\begin{lemma}
 $\mm{Sq}^1p_8= 0,$	$\mm{Sq}^ip_8\ne 0$ for $2\le i\le 8$.
\end{lemma}

For some elements of $C^\ast(\mm{PL})$, we can not directly get the $\mathscr A$ actions by Lemma \ref{K2primehomology}. 
Then, we will use the Adams spectral sequence (ASS) for ${_2}\pi_\ast(\mm{MSPL})$ to detect these $\mathscr A$ actions on $C^\ast(\mm{PL})$. 

\begin{lemma}\label{d2p9SPL}
Let $p_{9,1}$, $ p_{9,2}$ denote two generators of $C^9(\mm{PL})$. Then, $\mm{Sq}^1p_{9,1}=0$, $\mm{Sq}^1p_{9,2}\ne 0$.
\end{lemma}
\begin{proof} 
 It is easy to see that 
$$h_0^sp_8\in \mm{Ext}_{\mathscr A}^{s,s+8}(H^\ast(\mm{MSPL};\mb{Z}_2),\mb{Z}_2)\quad s\ge 0$$
$$p_{9,1},p_{9,2}\in \mm{Ext}_{\mathscr A}^{0,9}(H^\ast(\mm{MSPL};\mb{Z}_2),\mb{Z}_2).$$  
 Let $N^n=H^\ast(\mm{MSO})\otimes C^{\ge n}(\mm{PL})$. Applying the functor $\mm{Ext}_{\mathscr{A}}^{\ast,\ast}(-,\mb{Z}_2)$ on the short exact sequence $$0\to N^{n+1}\to N^n\to H^\ast(\mm{MSO})\otimes C^{ n}(\mm{PL})\to 0$$
 we get a long exact sequence for the $\mm{Ext}$-groups, thus an exact couple that establishes   
 the spectral sequence (AAHSS)
$$E_1^{s,m,n}(\mm{SPL})=\mm{Ext}_{\mathscr{A}}^{s,s+m}(H^\ast(\mm{MSO};\mb{Z}_2),\mb{Z}_2)\otimes C^n(\mm{PL})$$
\begin{equation}
	E_1^{s,m,n}(\mm{SPL})\Longrightarrow \mm{Ext}_{\mathscr{A}}^{s,s+m+n}(H^\ast(\mm{MSPL};\mb{Z}_2),\mb{Z}_2) \label{AAHSSSPL}
\end{equation}
Combining the results for $\mm{Ext}_{\mathscr A}^{s,t}(H^\ast(\mm{MSO}),\mb{Z}_2)$  \cite{Pen1982} and the above spectral sequence, we have
$\mm{Ext}_{\mathscr A}^{s,s+7}(H^\ast(\mm{MSPL};\mb{Z}_2),\mb{Z}_2)=0$ for $s>0$. 

In the Adams spectral sequence
\begin{equation}
\mm{Ext}_{\mathscr A}^{s,t}(H^\ast(\mm{MSPL};\mb{Z}_2),\mb{Z}_2)\Longrightarrow {_2\pi}_\ast(\mm{MSPL}),	\label{ASSMSPL}
\end{equation}
$d_r(p_8)=0$ for any $r\ge 2$. By $\Omega_8^{\mm{SPL}}=\mb{Z}^2\oplus\mb{Z}_4$ \cite{Will1966} and $\Omega_8^{\mm{SO}}=\mb{Z}^2$, we have $d_2(p_{9,1})=h_0^2p_8$ and $\mm{Sq}^1p_{9,1}=0$. By $\Omega_9^{\mm{SPL}}=\mb{Z}_2^3$ \cite{Will1966} and $\Omega_9^{\mm{SO}}=\mb{Z}_2^2$, we have that $\mm{Sq}^1p_{9,2}\ne 0$, and $p_{9,2}$ survives to the $E_\infty$-term.
\end{proof}

\begin{proposition}\label{C10}
	$C^{10}(\mm{PL})=\mb{Z}_2\langle p_{10,1}\rangle\oplus\mb{Z}_2\langle p_{10,2}\rangle \oplus\mb{Z}_2\langle p_{10,3}\rangle$ where $\mm{Sq}^2p_{8}=p_{10,1}$, $\mm{Sq}^1p_{9,2}=p_{10,2}$, and $\mm{Sq^1}p_{10,3}\ne 0$. 
\end{proposition}
\begin{proof}
Since $\mm{Sq}^2p_{8}\ne 0$ and $\mm{Sq}^1p_{9,2}\ne 0$, $C^{10}(\mm{PL})$ has generators so that $p_{10,1}=\mm{Sq}^2p_{8}$ and $p_{10,2}=\mm{Sq}^1p_{9,2}$. By Lemma \ref{K2primehomology}, we see that $p_{10,1}$ corresponds to $\mm{Sq}^2\ell_8\in H^{10}(\mm{K}(\mb{Z},8))\subset H^{10}(\mathscr{K})$. By Lemma \ref{K2primehomology}, $C^{10}(\mm{PL})$ has another generator $c_{10}$ that corresponds to $\ell_{10}\in H^{10}(\mm{K}(\mb{Z}_2,10))\subset H^{10}(\mathscr{K})$. Thus, $\mm{Sq}^1c_{10}\ne 0$. Since $\mm{Sq}^1p_{10,2}=\mm{Sq}^1\mm{Sq}^1p_{9,2}=0$, $c_{10}\ne p_{10,2}$. Moreover, the dual of $p_{10,2}$ is in $\mathscr{V}_{10}=\mb{Z}_2$ by Lemma \ref{VdePro}. 
\end{proof}

\begin{proposition}
	$\Omega_{10}^{\mm{SPL}}=\mb{Z}_2\oplus \mb{Z}_2^2$.
\end{proposition}
\begin{proof}
	By the AAHSS (\ref{AAHSSSPL}), $\mm{Ext}_{\mathscr A}^{s,s+9}(H^\ast(\mm{MSPL}),\mb{Z}_2)=0$ for any $s>0$. By Proposition \ref{C10}, $p_{10,1}$, $p_{10,3}\in \mm{Ext}_{\mathscr A}^{0,10}(H^\ast(\mm{MSPL}),\mb{Z}_2)$, thus survive to the $E_\infty^{0,10}$ in the ASS (\ref{ASSMSPL}). Since $\mm{Sq}^1p_{10,1}\ne 0$ and $\mm{Sq}^1p_{10,3}\ne 0$, we get the proposition.   
\end{proof}

\begin{proposition}\label{C11}
	$C^{11}(\mm{PL})=\mb{Z}_2\langle p_{11,1}\rangle\oplus\mb{Z}_2\langle p_{11,3}\rangle$ where $\mm{Sq}^1p_{10,1}=p_{11,1}$,  and $\mm{Sq^1}p_{10,3}=p_{11,3}$. 
\end{proposition}
\begin{proof}
	Since $V^{11}=0$ (see Lemma \ref{VdePro}), $C^{11}(\mm{PL})$ corresponds to $$H^{11}(\mathscr{K})=H^{11}(\mm{K}(\mb{Z},8)\times \mm{K}(\mb{Z}_2,10))=\mb{Z}_2^2.$$
	So the desired results follow from Proposition \ref{C10}.
\end{proof}

By Proposition \ref{C11}, the homomorphism
$$\mm{Ext}_{\mathscr A}^{s,s+11}(H^\ast(\mm{MSO}),\mb{Z}_2)\to \mm{Ext}_{\mathscr A}^{s,s+11}(H^\ast(\mm{MSPL}),\mb{Z}_2)$$
is isomorphic. By the AAHSS (\ref{AAHSSSPL}), $\mm{Ext}_{\mathscr A}^{s,s+10}(H^\ast(\mm{MSPL}),\mb{Z}_2)=0$ for any $s>0$.
 Since $\Omega_{11}^{\mm{SPL}}=\mb{Z}_2\oplus \mm{2gp}\oplus \mb{Z}_3$ \cite{Will1966}, we have
 \begin{corollary}
	$\Omega_{11}^{\mm{SPL}}=\mb{Z}_2\oplus \mb{Z}_3$.
\end{corollary}
Based on Lemma \ref{K2primehomology} and the proof of Proposition \ref{C11}, we have 
\begin{corollary}
	$\mm{Sq}^2p_{9,1}= 0$, $\mm{Sq}^2p_{9,2}= 0$.
\end{corollary}

Note that the Thom isomorphism $H_{\ast}(\mm{BSPL};\mb{Z})\cong H_{\ast}(\mm{MSPL};\mb{Z})$. Finally, we consider $H_{\ast}(\mm{BSPL};\mb{Z})$ by the spectral sequence (AHSS)
\begin{equation}
	E_2^{p,q}(\mm{SPL})=H_p(\mm{MSPL};\pi_q(S^0))\Longrightarrow \pi_{p+q}(\mm{MSPL}) \label{MSPLAHSS}
\end{equation}
where $S^0$ is the sphere spectrum.
\begin{lemma}\label{BSPLhomologyZ}
	$H_{n}(\mm{BSPL};\mb{Z})$ does not have odd torsion for $n<11$. 
\end{lemma}
\begin{proof}
	It is well known that $H_{n}(\mm{BSO};\mb{Z})\cong H_{n}(\mm{BSPL};\mb{Z})$ for $n\le 7$, and $H_{\ast}(\mm{BSO};\mb{Z})$ does not have odd torsion.
Recall that, in dimensions $i \le 11$ \cite{Toda}, $\pi_i(S^0)$ is as follows.

\begin{center}
\begin{tabular}{|c|c|c|c|c|c|c|c|c|c|c|c|c|c|c|c|}
\hline $i$ & 0&1& 2& 3&4& 5& 6&7&8&9&10&11\\
\hline
$\pi_i(S^0 )$ & $\mathbb{Z}$ & $\mathbb{Z}_2$ & $\mathbb{Z}_2$ & $\mathbb{Z}_{24}$& 0& 0& $\mathbb{Z}_2$ & $\mm{Z}_{240}$
& $ \mathbb{Z}_2^2$ & $\mathbb{Z}^3_2$&$\mb{Z}_6$&$\mb{Z}_{504}$\\
\hline
\end{tabular}
\end{center}

Note that $\mm{PL/O}$ is $6$-connected, $\pi_7(\mm{PL/O})=H_{7}(\mm{PL/O};\mb{Z})=\mb{Z}_{28}$. By the Serre spectral sequence for the fibration,
$$\mm{PL/O}\to \mm{BSO}\to \mm{BSPL}$$
$H_{8}(\mm{BSPL};\mb{Z})$ does not have any element with finite order $i$ so that $i$ is prime to $2$, $7$; and $H_{8}(\mm{BSO};\mb{Z}_3)\to H_{8}(\mm{BSPL};\mb{Z}_3)$ is isomorphic.

Note that for $p<7$, $E_2^{p,7-p}(\mm{SPL})$ is a finite group that does not have any element with order $i$ so that $(i,7)=7$. 
By $\pi_8(\mm{MSPL})=\mb{Z}^2\oplus \mb{Z}_4$ \cite{Will1966} and the differentials of the spectral sequence (\ref{MSPLAHSS}), $H_{8}(\mm{BSPL};\mb{Z})$ does not have odd torsion. 

By the same arguments, the desired results for $n=9,10$ follow.
\end{proof}

\section{$\Omega_\ast^{\mm{PL}\langle 8\rangle}$ cobordism}
We first consider the integral homology group for $\mm{BPL}\langle n\rangle$. 
 \begin{lemma}\label{BPL8homologyZ}
	Let $n=4$, $8$.	$H_{i}(\mm{BPL}\langle n\rangle;\mb{Z})$ does not have odd torsion for $i<11$. 
\end{lemma}
\begin{proof}
	Recall the fibrations   
  $$\mm{K}(\mb{Z}_2,1)\to \mm{BPL}\langle 4\rangle\to \mm{BSPL}$$
  $$\mm{K}(\mb{Z},3)\to \mm{BPL}\langle 8\rangle\to \mm{BPL}\langle 4\rangle$$
  Since $H^i(\mm{K}(\mb{Z}_2,1);\mb{Z})=\mb{Z}_2$ or $0$ for any $i> 0$, the homomorphism
 $$ H^\ast(\mm{BSPL};\mb{Z}_{p^r})\to  H^\ast(\mm{BPL}\langle 4\rangle;\mb{Z}_{p^r})$$
is isomorphic for any $r>0$ and any odd prime number $p$. The desired results for $n=4$ follow by Lemmma \ref{BSPLhomologyZ}. 

By the conclusions for $n=4$, we compute the case for $n=8$. Recall that, in dimensions $\le 11$, the nontrivial integral cohomology groups of $\mm{K}(\mathbb{Z}, 3)$ are as follows 
\begin{center}
\begin{tabular}{|c|c|c|c|c|c|c|c|c|c|c|c|c|c|c|c|}
\hline  
 i & 3& 6 &8 &9 &10 &11 \\
\hline  
$H^i$ &$\mathbb{Z}\langle l_3 \rangle$ &$\mathbb{Z}_2\langle l_3^2 \rangle$&$\mathbb{Z}_3\langle a_8 \rangle$&$\mathbb{Z}_2\langle l^3_3 \rangle$&$\mathbb{Z}_2\langle a_{10} \rangle$&$\mathbb{Z}_3\langle a_8l_3 \rangle$\\
\hline
\end{tabular}
\end{center}
For the $7$-connected covers of $\mm{BO}$ and $\mm{BPL}$, we have 
  \[
\xymatrix@C=0.8cm{
\mm{K}(\mb{Z},3)\ar[d]^-{}\ar@{=}[rr]&&\mm{K}(\mb{Z},3)\ar[d]^-{}\ar@{=}[rr]&&\mm{K}(\mb{Z},3)\ar[d]^-{}\\
\mm{BO}\langle 8\rangle \ar[rr]^-{\mathscr{D}\langle 8\rangle}\ar[d]^-{p_1}& & \mathrm{BPL}\langle 8\rangle\ar[rr] \ar[d]^-{p_{2}}&&\mm{PK}(\mb{Z},4)\ar[d]^-{}\\
  \mm{BO}\langle 4\rangle \ar[rr]^-{\mathscr{D}\langle 4\rangle}&& \mm{BPL}\langle 4\rangle \ar[rr]^-{g_{l}}&&\mm{K}(\mb{Z},4)
}
\]
 As the proof of Lemma \ref{BSPLhomologyZ}, $H_{8}(\mm{BO}\langle 4\rangle ;\mb{Z}_3)\to H_{8}(\mm{BPL}\langle 4\rangle ;\mb{Z}_3)$ is isomorphic. Note $H_{10}(\mm{BO}\langle 8\rangle;\mb{Z}_3)=0$ \cite{Giambalvo}.
Considering these $E_2^{p,q}(\mm{PL}8)$-terms for $8\le p+q\le 10$ and the following morphism
\[
\xymatrix@C=.2cm{
  {E}_2^{p,q}(\mm{O}8)=H_{p}(\mm{BO}\langle 4\rangle;H_q(\mm{K}(\mb{Z},3);\mb{Z}))\ar@{=>}[rr]^{}\ar[d]&  & H_{p+q}(\mm{BO}\langle 8\rangle;\mb{Z})\ar[d]^-{} \\
{E}_2^{p,q}(\mm{PL}8)=H_{p}(\mm{BPL}\langle 4\rangle;H_q(\mm{K}(\mb{Z},3);\mb{Z}))\ar@{=>}[rr]^-{}& & H_{p+q} (\mm{BPL}\langle 8\rangle;\mb{Z})
}
\]
we have the desired results for the case $n=8$.
\end{proof}

Then, we compute the $\mb{Z}_2$-cohomology of $\mm{BPL}\langle 8\rangle$.   
\begin{lemma}\label{BPLandBO}
		$H^\ast(\mm{BPL}\langle 8\rangle)\cong H^\ast(\mm{BO}\langle 8\rangle)\otimes  C^\ast(\mm{PL})$ as right $\mathscr{A}$-modules. The formula $($\ref{rightmodule}$)$ induces the right $\mathscr{A}$ action on $H^\ast(\mm{BO}\langle 8\rangle)\otimes  C^\ast(\mm{PL}).$
\end{lemma}
\begin{proof}
Let $g_o=\mathscr D\langle 4\rangle \circ g_l$. Then $g_o^\ast(\ell_4)=w_4$, where $\ell_4$ and
 $w_4$ are the generators of $H^4(\mm{K}(\mb{Z},4))$ and $H^4(\mm{BO}\langle 4\rangle)$. Recall that $H^\ast(\mm{K}(\mathbb{Z},4))$ is a polynomial ring with generators $\ell_4$ and $\mathrm{Sq}^\mm{I} \ell_4$, where $\mm{I}=(i_1,i_2,\cdots,i_r)$ is the admissible sequence satisfying
$ e(\mm{I})=2i_1-(i_1+\cdots +i_r )< 4$
and $i_r>1.$  By \cite{Stong1963}, 
$g_o^\ast:H^\ast(\mm{K}(\mathbb{Z},4))\to H^\ast(\mm{BO}\langle 4\rangle)$ is a monomorphism. Thus, $g_l^\ast:H^\ast(\mm{K}(\mathbb{Z},4))\to H^\ast(\mm{BPL}\langle 4\rangle)$ is also a monomorphism.
By the case $n=4$ of Theorem \ref{PLtoBO248}, all differentials of the spectral sequence
$${E}_2^{m+p,q}=H^{m+p}(\mm{BPL}\langle 4\rangle;H^q(\mm{K}(\mb{Z},3)))\Longrightarrow H^{m+p+q}(\mm{BPL}\langle 8\rangle)$$
are determined by the differentials of the  following spectral sequences
$$E_2^{p,q}=H^{p}(\mm{K}(\mb{Z},4);H^q(\mm{K}(\mb{Z},3)))\Longrightarrow H^{p+q}(\mm{PK}(\mb{Z},4))$$
$$E_2^{p,q}=H^{p}(\mm{BO}\langle 4\rangle;H^q(\mm{K}(\mb{Z},3)))\Longrightarrow H^{p+q}(\mm{BO}\langle 8\rangle).$$
Thus $H^\ast(\mm{BPL}\langle 8\rangle)\cong H^\ast(\mm{BO}\langle 8\rangle)\otimes C^\ast(\mm{PL})$
as vector spaces. The desired $\mathscr{A}$-module isomorphism follows by the commutative diagram
\[
\xymatrix@C=1cm{
H^\ast(\mm{BO}\langle 4\rangle)\otimes C^\ast(\mm{PL}) \ar[rr]^-{\cong}_-{right-\mathscr{A}-module}\ar[d]^-{p_1^\ast\otimes 1}_-{right-\mathscr{A}-module}& & H^\ast(\mathrm{BPL}\langle 4\rangle) \ar[d]_-{p_{2}^\ast}^-{right-\mathscr{A}-module}\\
  H^\ast(\mm{BO}\langle 8\rangle)\otimes C^\ast(\mm{PL}) \ar[rr]^-{\cong}&& H^\ast(\mm{BPL}\langle 8\rangle) 
  }
\]
where $p_1^\ast$ is an epimorphism \cite{Giambalvo}.  
\end{proof}

Indeed, by formula (\ref{rightaction}), the right $\mathscr{A}$ action on $H^\ast(\mm{BPL}\langle n\rangle
)$ is induced by the canonical left $\mathscr{A}$ action on $H^\ast(\mm{MPL}\langle n\rangle)$, where $\mm{MPL}\langle n\rangle$ is the Thom spectrum associated with the universal bundle over $\mm{BPL}\langle n\rangle$. 
\begin{corollary}\label{leftMPL}
$H^\ast(\mm{MPL}\langle n\rangle)\cong H^\ast(\mm{MO}\langle n\rangle)\otimes  C^\ast(\mm{PL})$ as canonical left $\mathscr{A}$-modules for $n=2,4,8$, and the formula $$\alpha  (b\otimes c)=\Sigma   \alpha^\prime b \otimes \alpha^{\prime \prime}c \quad \Delta(\alpha)=\Sigma\alpha^\prime \otimes \alpha^{\prime \prime}$$ presents the left $\mathscr{A}$ action on $H^\ast(\mm{MO}\langle n\rangle)\otimes  C^\ast(\mm{PL})$.\end{corollary}

Now, we discuss some representatives of $\Omega^{\mm{O}\langle 8\rangle}_\ast=\pi_\ast(\mm{MO} \langle  8 \rangle )$.
Recall that, in dimension $i \le 14$ \cite{Giambalvo,HoRa1995}, $\pi_i(\mm{MO} \langle  8 \rangle )$ is as follows.

\begin{center}
\begin{tabular}{|c|c|c|c|c|c|c|c|c|c|c|c|c|c|c|c|}
\hline
$i$ & 0&1& 2& 3&4& 5& 6&7&8&9\\
\hline
$\pi_i(\mm{MO} \langle  8\rangle )$ & $\mathbb{Z}$ & $\mathbb{Z}_2$ & $\mathbb{Z}_2$ & $\mathbb{Z}_{24}$& 0& 0& $\mathbb{Z}_2$ & 0& $ \mathbb{Z}\oplus \mathbb{Z}_2$ & $\mathbb{Z}_2\oplus \mathbb{Z}_2 $\\
\hline

\end{tabular}
\end{center}
\begin{center}
\begin{tabular}{|c|c|c|c|c|c|c|c|c|c|c|c|c|c|c|c|}
\hline
$i$ &10&11&12&13&14\\
\hline
$\pi_i(\mm{MO} \langle  8 \rangle )$  & $\mathbb{Z}_6 $ & 0& $\mathbb{Z}$ & $\mathbb{Z}_3$ & $\mathbb{Z}_2$\\
\hline

\end{tabular}
\end{center}

\begin{lemma}\label{exoticsphere}
	Any torsion element $\sigma_n\in \pi_n(\mm{MO} \langle  8 \rangle)$ for $n=8,9,10$ is represented by an exotic sphere $\Sigma^n$.
\end{lemma}
\begin{proof}
By \cite{Kerva1960} and \cite{KervaMilnor}, we have the exact sequence
$$0\to \mm{bP}_{n+1}\to \Theta_n\to coker({J}_n)\to 0$$
for $ n=8,9,10$ where $J_n:\pi_n(\mm{SO})\to \pi_n(S^0)$. By \cite{Adams1966} and \cite{Quillen}, $coker(J_n)$ has the following generators
$$\phi \in \mathbb{Z}_2\subset \pi_8(S^0),\quad \eta\phi  \in \mathbb{Z}_2\subset \pi_9(S^0),$$
$$ \kappa  \in \mathbb{Z}_2\subset \pi_9(S^0), \quad \eta\kappa  \in \mathbb{Z}_2\subset \pi_{10}(S^0), \quad \gamma  \in \mathbb{Z}_3\subset \pi_{10}(S^0)$$
where $\eta\in \pi_1(S^0)$ is represented by $h_1$ in the ASS for $S^0$. Note that exotic sphere $\Sigma^n$ has the trivial $\mm{BO}\langle 8\rangle$-structure. Considering the unit $\iota:S^0\to \mm{MO}\langle 8 \rangle$, we have that $\iota_\ast (\phi)=c_0$, $\iota_\ast(\kappa)=h_1\omega $, $\pi_{10}(S^0)\cong \pi_{10}(\mm{MO}\langle 8\rangle)$ 
by \cite{Giambalvo,HoRa1995} and \cite{Ravenel1986}.
	Thus, the lemma follows.
\end{proof}

Next, we consider the ASS for ${_2}\pi_\ast(\mm{MPL}\langle 8\rangle)$. 
Recall the $E_2$-terms of the Adams spectral sequence (ASS) for $\mm{MO} \langle  8 \rangle $ as in Fig.\ref{Fig.1}. From \cite{Giambalvo}, we have $H^\ast(\mm{MO}\langle 8\rangle)\cong \mathscr{A}//\mathscr{A}_2$ as modules over the Steenrod algebra $\mathscr A$ in dimensions $<20$, thus
$$\mm{Ext}^{s,t}_{\mathscr{A}}(H^\ast(\mm{MO}\langle 8\rangle), \mathbb{Z}_2)\cong \mm{Ext}^{s,t}_{\mathscr{A}_2}(\mb{Z}_2, \mathbb{Z}_2)$$
for $t-s\le 17$, where $\mathscr{A}_2$ is a subalgebra of $\mathscr A$, generated by $\mm{Sq}^1$, $\mm{Sq}^2$ and $\mm{Sq}^4$.
By Corollary \ref{leftMPL} and the change of ring,
$$\mm{Ext}^{s,t}_{\mathscr{A}}(H^\ast(\mm{MPL}\langle 8\rangle), \mathbb{Z}_2)\cong \mm{Ext}^{s,t}_{\mathscr{A}_2}(C^\ast(\mm{PL}), \mathbb{Z}_2)\quad t-s\le 17.$$
\begin{figure}
\centering
\begin{tikzpicture}[>=stealth,scale=1,line width=0.5pt]{enumerate}

\pgfmathsetmacro{\ticker}{0.125}

\coordinate [label=225:$0$](A) at (0,0);
\coordinate (B) at (0,5.6);
\coordinate (D) at (10.2,0);
\draw(D)--(A)--(B);

\coordinate [label=left:$s$](E) at ($(B)+(-0.4,-0.2)$);
\coordinate [label=below:$t-s$](F) at ($(D)+(-0.2,-0.4)$);

\foreach \i/\texti  in {1,2,3,4,5,6,7,8,9,10,11,12,13,14,15,16,17} {
\draw (0.6*\i,0) --(0.6*\i,\ticker) node[label=below:\texti]{};
}
\foreach \j/\textj  in {1,2,3,4,5,6,7,8} {
\draw (0,0.7*\j) --(\ticker,0.7*\j) node[label=left:\textj]{};
}
{\tiny
\coordinate[label=right:$h_0$] (I) at (0.1,0.7);
}
{\tiny
\coordinate[label=right:$h^2_0$] (I) at (0.1,1.4);
}
\path[name path = row1](0,0.7)--(10.2,0.7);
\path[name path = row2](0,1.4)--(10.2,1.4);
\draw[name path = row3][dashed] (0,2.1) -- (10.2,2.1);
\draw[name path = xieh1](0,0)--(0.6,0.7)--(1.2,1.4)--(1.8,2.1);
\path[draw,fill,name intersections={of = row1 and xieh1,by=h1}](h1)circle(1pt);
{\tiny
\coordinate[label=right:$h_1$] (I) at (0.7,0.7);
}
\path[draw,fill,name intersections={of = row2 and xieh1,by=h12}](h12)circle(1pt);
{\tiny
\coordinate[label=left:$h^2_1$] (I) at (1.1,1.4);
}
\path[draw,fill,name intersections={of = row3 and xieh1,by=h13}](h13)circle(1pt);
{\tiny
\coordinate[label=above:${h^3_1=h^2_0h_2}$] (I) at (1.8,2.1);
}
\draw[name path = c3] (1.8,0.7) -- (h13);
\path[draw,fill,name intersections={of = c3 and row2,by=h0h2}](h0h2)circle(1pt);
\path[draw,fill,name intersections={of = c3 and row1,by=h2}](h2)circle(1pt);
{\tiny
\coordinate[label=left:${h_2}$] (I) at (1.7,0.8);
}
\draw[name path = xieh2](0,0)--(h2)--(3.6,1.4);
\path[draw,fill,name intersections={of = xieh2 and row2,by=h22}](h22)circle(1pt);
{\tiny
\coordinate[label=left:${h^2_2}$] (I) at (3.5,1.5);
}
\path[name path = row4](0,2.8)--(10.2,2.8);
\path[name path = row5](0,3.5)--(10.2,3.5);
\path[name path = row6](0,4.2)--(10.2,4.2);
\path[name path = row7](0,4.9)--(10.2,4.9);
\path[name path = row8](0,5.6)--(10.2,5.6);
\path[name path = c8](4.8,0)--(4.8,5.6);
\path[draw,fill,name intersections={of = c8 and row3,by=c0}](c0)circle(1pt);
{\tiny
\coordinate[label=left:${c_0}$] (I) at (4.7,2.0);
}
\path[draw,fill,name intersections={of = xieh2 and row2,by=h22}](h22)circle(1pt);
\draw(4.8,2.8)--(4.8,5.6);
\path[draw,fill,name intersections={of = c8 and row4,by=omega}](omega)circle(1pt);
{\tiny
\coordinate[label=left:${\omega}$] (I) at (4.7,2.8);
}
\path[draw,fill,name intersections={of = c8 and row5,by=h0omega}](h0omega)circle(1pt);
\path[draw,fill,name intersections={of = c8 and row6,by=h02omega}](h02omega)circle(1pt);
\path[draw,fill,name intersections={of = c8 and row7,by=h03omega}](h03omega)circle(1pt);
\path[draw,fill,name intersections={of = c8 and row8,by=h04omega}](h04omega)circle(1pt);
\draw[name path = xiec0](4.8,2.1)--(5.4,2.8);
\path[draw,fill,name intersections={of = row4 and xiec0,by=h1c0}](h1c0)circle(1pt);
\draw[name path = xieomega](4.8,2.8)--(6.6,4.9);
\path[draw,fill,name intersections={of = row5 and xieomega,by=h1omega}](h1omega)circle(1pt);
\path[draw,fill,name intersections={of = row6 and xieomega,by=h12omega}](h12omega)circle(1pt);
\path[draw,fill,name intersections={of = row7 and xieomega,by=h13omega}](h13omega)circle(1pt);
{\tiny
\coordinate[label=above:${h_1^3\omega=h^2_0h_2\omega}$] (I) at (6.3,4.9);
}
\draw[name path = c11](h13omega)--(6.6,3.5);
\path[draw,fill,name intersections={of = row5 and c11,by=h2omega}](h2omega)circle(1pt);
\path[draw,fill,name intersections={of = row6 and c11,by=h0h2omega}](h0h2omega)circle(1pt);
{\tiny
\coordinate[label=left:${h_2\omega}$] (I) at (6.6,3.6);
}
\draw[name path = xie2omega](omega)--(8.4,4.2);
\path[draw,fill,name intersections={of = row6 and xie2omega,by=h22omega}](h22omega)circle(1pt);
{\tiny
\coordinate[label=above:${h^2_2\omega=h^2_0d_0}$] (I) at (h22omega);
}
\draw[name path = c12](7.2,2.1)--(7.2,5.6);
\path[draw,fill,name intersections={of = row3 and c12,by=tau}](tau)circle(1pt);
\path[draw,fill,name intersections={of = row4 and c12,by=h0tau}](h0tau)circle(1pt);
\path[draw,fill,name intersections={of = row5 and c12,by=h02tau}](h02tau)circle(1pt);
\path[draw,fill,name intersections={of = row6 and c12,by=h03tau}](h03tau)circle(1pt);
\path[draw,fill,name intersections={of = row7 and c12,by=h04tau}](h04tau)circle(1pt);
\path[draw,fill,name intersections={of = row8 and c12,by=h05tau}](h05tau)circle(1pt);
{\tiny
\coordinate[label=left:${\tau}$] (I) at (7.1,2);
}
\draw[name path = c14](h22omega)--(8.4,2.8);
\path[draw,fill,name intersections={of = row4 and c14,by=d0}](d0)circle(1pt);
\path[draw,fill,name intersections={of = row5 and c14,by=h0d0}](h0d0)circle(1pt);
{\tiny
\coordinate[label=below:${d_0}$] (I) at (8.3,2.8);
}
\draw[name path = xied0](d0)--(9,3.5);
{\tiny
\coordinate[label=right:${h^2_0\kappa=h_1d_0}$] (I) at (9,3.5);
}
\path[draw,fill,name intersections={of = row5 and xied0,by=h1d0}](h1d0)circle(1pt);
\draw[name path = c15](h1d0)--(9,2.1);
\path[draw,fill,name intersections={of = row3 and c15,by=kappa}](kappa)circle(1pt);
\path[draw,fill,name intersections={of = row4 and c15,by=h0kappa}](h0kappa)circle(1pt);
{\tiny
\coordinate[label=left:${\kappa}$] (I) at (8.9,2);
}
\draw[dashed](7.2,2.1)--(6.6,3.5);
\draw[dashed](7.2,2.8)--(6.6,4.2);\draw[dashed](7.2,3.5)--(6.6,4.9);
\draw[dashed](9,2.1)--(8.4,3.5);\draw[dashed](9,2.8)--(8.4,4.2);

\end{tikzpicture}
\renewcommand{\figurename}{Fig}
 \caption{ $\mm{Ext}^{s,t}_{\mathscr{A}}(H^\ast(\mm{MO}\langle 8\rangle), \mathbb{Z}_2)$}
  \label{Fig.1}
\end{figure}
For $\mm{Ext}^{s,t}_{\mathscr{A}_2}(C^\ast(\mm{PL}), \mathbb{Z}_2)$,  
there is an AAHSS 
$$E_1^{s,m,n}(\mm{PL}8)=\mm{Ext}^{s,s+m}_{\mathscr{A}_2}(\mb{Z}_2, \mathbb{Z}_2)\otimes C^n(\mm{PL})$$
\begin{equation}
	E_1^{s,m,n}(\mm{PL}8)\Longrightarrow \mm{Ext}_{\mathscr{A}_2}^{s,s+m+n}(C^\ast(\mm{PL}),\mb{Z}_2) \label{AAHSSPL8}
\end{equation}
$$d_r:E_r^{s,m,n}\to E_r^{s+1,m+r-1,n-r}$$

\begin{lemma}\label{MOtoPLExt}
	The natural homomorphism 
	$$\mathscr D\langle 8\rangle_\ast :\mm{Ext}^{s,t}_{\mathscr{A}}(H^\ast(\mm{MO}\langle 8\rangle), \mathbb{Z}_2)\to \mm{Ext}^{s,t}_{\mathscr{A}}(H^\ast(\mm{MPL}\langle 8\rangle), \mathbb{Z}_2)$$
	is injective for $t-s\le 11$.
\end{lemma}
\begin{proof}
	Since $\mathscr D\langle 8\rangle:\mm{MO}\langle 8\rangle\to \mm{MPL}\langle 8\rangle$ is $7$-equivalent, it is easy to get the lemma for $t-s\le 7$.
	
Note that $C^n(\mm{PL})=0$ for $n<8$. By the degrees of differentials of the AAHSS (\ref{AAHSSPL8}), the lemma follows	for $8\le t-s\le 11$. 
\end{proof}

By the $\mm{Sq}^1$, $\mm{Sq}^2$ and $\mm{Sq}^4$ actions on $C^\ast(\mm{PL})$ (shown in Section \ref{CPL}) and the AAHSS (\ref{AAHSSPL8}), we have
\begin{lemma}\label{ExtMPL8}
	Partial cokernel of $\mathscr D\langle 8\rangle_\ast$ is shown as follows 
\end{lemma}
\begin{center}
\begin{tabular}{|c|c|c|c|c|c|c|c|c|c|c|c|c|c|c|c|}
\hline
$t-s$ & 8&9& 9& 10&10& 11\\
\hline
s&$\ge$ 0&$\ge$0&0&0&1&2\\
\hline
cok$(\mathscr D\langle 8\rangle_\ast)$ & $h_0^sp_8$ & $h_0^sp_{9,1}$ & $p_{9,2}$ & $p_{10,3}$& $h_1p_{9,1}$, $h_1p_{9,2}$& $h^2_1p_{9,1}$, $h^2_1p_{9,2}$\\
\hline
\end{tabular}
\end{center}	

\begin{lemma}\label{2priPL8}
	In the $\mm{ASS}$
	\begin{equation}
		\mm{Ext}^{s,t}_{\mathscr{A}}(\mm{PL}\langle 8\rangle)=\mm{Ext}^{s,t}_{\mathscr{A}}(H^\ast(\mm{MPL}\langle 8\rangle), \mathbb{Z}_2)\Longrightarrow {_2\pi}_{t-s}(\mm{MPL}\langle 8\rangle) \label{ASSMPL8}
	\end{equation}
	$d_2(p_{9,1})=h^2_0p_8$, 
	$d_3(p_{9,2})=c_0$, $d_4(h_1p_{9,1})=h_1\omega$.
\end{lemma}
\begin{proof}
	It is easy to see that $d_2(p_{9,1})=h^2_0p_8$ follows from the proof of Lemma \ref{d2p9SPL} and the map $\mm{MPL}\langle 8\rangle\to \mm{MSPL}$. 
	
Since that $c_0\in \pi_8(\mm{MO}\langle 8\rangle)$ denotes an exotic sphere (see Lemma \ref{exoticsphere}), and $c_0\in \mm{Ext}^{3,11}_{\mathscr{A}}(\mm{PL}\langle 8\rangle)$ (see Lemma \ref{MOtoPLExt}), we have that $c_0$ must be killed by a differential in the ASS (\ref{ASSMPL8}). So $d_3(p_{9,2})=c_0$. 

Similarly, $h_1^s\omega$ for $s=1,2$ must also be killed by certain differentials. By $h_1p_{10,3}=0\in \mm{Ext}^{1,12}_{\mathscr{A}}(\mm{PL}\langle 8\rangle)$, $d_5(p_{10,3})\ne h_1\omega$. Since $h^2_0h_1p_8=0\in \mm{Ext}^{3,12}_{\mathscr{A}}(\mm{PL}\langle 8\rangle)$, 
$d_2(h_1p_{9,1})=0$. So $d_4(h_1p_{9,1})=h_1\omega$.
\end{proof}


\begin{proposition}\label{MPL8homotopy}
	${\pi}_{8}(\mm{MPL}\langle 8\rangle)=\mb{Z}_4\oplus \mb{Z},$ 
	
	${\pi}_{9}(\mm{MPL}\langle 8\rangle)=0,$ 
	${\pi}_{10}(\mm{MPL}\langle 8\rangle)=\mb{Z}_2.$ 
\end{proposition}	
\begin{proof}
	By Lemma \ref{MOtoPLExt}, \ref{ExtMPL8} and \ref{2priPL8}, we get the ${_2\pi}_n(\mm{MPL}\langle 8\rangle)$. By Lemma \ref{BPL8homologyZ} and the spectral sequence $$E_2^{p,q}=H_p(\mm{MPL}\langle 8\rangle;\pi_q(S^0))\Longrightarrow \pi_{p+q}(\mm{MPL}\langle 8\rangle)$$
	 we see that ${\pi}_{n}(\mm{MPL}\langle 8\rangle)$ does not have odd torsion for $n=8,9$.
	 \[
\xymatrix@C=.2cm{
  {E}_2^{p,q}(\mm{O}8)=H_p(\mm{MO}\langle 8\rangle ;\pi_q(S^0))\ar@{=>}[rr]^{}\ar[d]&  & \pi_{p+q}(\mm{MO}\langle 8\rangle)\ar[d]^-{} \\
{E}_2^{p,q}=H_p(\mm{MPL}\langle 8\rangle;\pi_q(S^0))\ar@{=>}[rr]^-{}& & \pi_{p+q} (\mm{MPL}\langle 8\rangle)
}
\]
Note that $\pi_{10}(\mm{MO}\langle 8\rangle)={E}_\infty^{0,10}(\mm{O}8)={E}_2^{0,10}(\mm{O}8)=\mb{Z}_6$. By Lemma \ref{exoticsphere}, $\mathscr D \langle 8\rangle_\ast({E}_\infty^{0,10}(\mm{O}8))=0$. Thus, $ {E}_\infty^{0,10}=0$. By Lemma \ref{BPL8homologyZ} and the spectral sequence, ${\pi}_{10}(\mm{MPL}\langle 8\rangle)$ does not have odd torsion. 
\end{proof}

It is very difficult to compute ${\pi}_{11}(\mm{MPL}\langle 8\rangle),{\pi}_{13}(\mm{MPL}\langle 8\rangle)$ by the algebraic method. Next, we consider them by surgery theory.

Recall the first stability theorem  
\begin{theorem}\cite{KirbSieb}
	Let $\mm{CAT}=\mm{O}$ or $\mm{PL}$. For $i\le m+1$, $m\ge 5$,
	$$\pi_i(\mm{TOP/CAT},\mm{TOP_m/CAT_m})=0.$$
\end{theorem}
As $\pi_{i}(\mm{O}_m)\cong \pi_i(\mm{O})$ for $i\le m-2$, we have
\begin{corollary}\label{stable}
	$\pi_{i}(\mm{TOP_m})\cong\pi_{i}(\mm{TOP})$, $\pi_{i}(\mm{PL_m})\cong\pi_{i}(\mm{PL})$ for $4\le i\le m-2$.
\end{corollary}
\begin{proof}
	Consider the following exact sequences
\[
\xymatrix@C=0.7cm{
\cdots \ar[r]^-{}\ar[d]^-{}&\pi_{i}(\mm{CAT}_m)\ar[r]\ar[d] & \pi_{i}(\mm{TOP_m})\ar[r] \ar[d]^-{}&\pi_{i}(\mm{TOP_m/CAT_m})\ar[d]\ar[r]&\cdots\\
  \cdots \ar[r]^-{}&  \pi_{i}(\mm{CAT})\ar[r]&  \pi_{i}(\mm{TOP})\ar[r]&  \pi_{i}(\mm{TOP/CAT})\ar[r]&\cdots
}
\]	
If $\mm{CAT}=\mm{O}$, $\pi_{i}(\mm{TOP_m})=\pi_{i}(\mm{TOP})$ for $4\le i\le m-2$ by the Five-Lemma. For $m\ge 5,$ $\mm{TOP_m/PL_m}=\mm{K}(\mb{Z}_2,3)$ \cite{KirbSieb}. Hence $\pi_{i}(\mm{PL_m})=\pi_{i}(\mm{PL})$ for $4\le i\le m-2$. 
\end{proof}

\begin{proposition}\label{13MPL8}
	${\pi}_{11}(\mm{MPL}\langle 8\rangle)={\pi}_{13}(\mm{MPL}\langle 8\rangle)=0.$
\end{proposition}
\begin{proof}
Let $n=5$ or $6$. For any bordism class $$\{M\}\in {\Omega}_{2n+1}^{\mm{PL}\langle 8\rangle}\cong {\pi}_{2n+1}(\mm{MPL}\langle 8\rangle),$$ following the surgery theory, we can assume that $H_{i}(M;\mb{Z})$ is nontrivial only for
	 $i=0,n,n+1,2n+1$. Let $f:M\to \mm{BPL}\langle 8\rangle$ be a $\mm{BPL}\langle 8\rangle$-structure of $M$.
	 Since $\pi_n(M)\cong H_n(M;\mb{Z})$, any generator $x\in H_n(M;\mb{Z})$ can be represented by a PL embedding $g:S^n\to M$ \cite{RourkeSanderson}. Note that $g\circ f\simeq 0$.
By Corollary \ref{stable}, $\pi_n(\mm{BPL}_{n+1})=0$. The embedding $g$ can be represented by the PL embedding $\bar g:S^n\times D^{n+1}\to M$. Let $F$ be the homotopy fibre of the map $\mm{BPL}\langle 8\rangle\to \mm{BPL}$. Since $\pi_n(F)\cong \pi_{n+1}(\mm{BPL})\cong \pi_{n+1}(\mm{BO})=0$, the $\mm{BPL}\langle 8\rangle$-structure of $S^n$ is unique. Hence we can kill the elements in $H_n(M;\mb{Z})$ by surgery. In other words, $M$ is $\mm{BPL}\langle 8\rangle$-bordant to $S^{2n+1}$. So  
${\Omega}_{2n+1}^{\mm{PL}\langle 8\rangle}=0$.
\end{proof}

\section{Kervaire invariant}
 
By the basic results of \cite{KirbSieb}, we have a PL version of the Theorem 1 of \cite{Freed} or the Corollary 4 in Section 7 of \cite{Kreck1999}.
\begin{theorem}\label{PLhome}
	Let $\xi$ be a $\mm{PL}$ linear bundle over a simply connected $\mm{CW}$ complex $X$. Let $M^{2n}_i\stackrel{f_i}{\to }X$, $i=0,1$, $n\ge 5$, be normal maps from closed $\mm{PL}$ manifolds, i.e. $f^\ast_i \xi=\nu(M_i)$, the stable normal bundle over $M_i$. Suppose that $n$ is odd, $f_0$ and $f_1$ are normally bordant, $f_i$ is $n$-connected, and that $B_n(M_0) = B_n(M_1)$ where $B_n$ denotes the $n$-th Betti number. Then $M_0$ and $M_1$ are $\mm{PL}$-homeomorphic.    
\end{theorem} 

Here, we take $X=\mm{BPL}\langle 8\rangle$ associated with the universal PL bundle $\xi=\gamma_8^{pl}$.
Let $M^{10}$ be a closed 4-connected manifold. So $H^5(M;\mb{Z})$ is free abelian of even rank $2s$. By the fibration 
$$\mm{BPL}\to \mm{BTOP}\to \mm{K}(\mb{Z}_2,4)$$
$M$ admits a unique PL-manifold structure.  
Moreover, by obstruction theory, there exists a normal map $f:M\to \mm{BPL}\langle 8\rangle$, unique up to homotopy. 
Then, by Theorem \ref{PLhome}, we have 

\begin{corollary}\label{bor-home}
	Let $M^{10}_i$, $i=0,1$ be closed $4$-connected manifolds with the same $B_5(M_i)=2s$. They are $\mm{PL}$-homeomorphic if $M_0$ is normally $\mm{BPL}\langle 8\rangle$-bordant to $M_1$. 
\end{corollary}


In \cite{Kerva1960}, Kervaire presented an invariant for such closed 4-connected manifolds. Denote it by $\Phi(M)\in \mb{Z}_2$.




\begin{lemma}
	If a $4$-connected closed manifold $M$ of dimension $10$ is $\mm{BPL}\langle 8\rangle$-bordant to $0$, then $\Phi(M)=0$.
\end{lemma}
\begin{proof}
	The assumption implies the existence of a $\mm{BPL}\langle 8\rangle$-manifold $V^{11}$ with boundary $M$. By classic surgery theory, we can find a new $V$ so that it is 4-connected. Then the lemma follows from the proof of Lemma 2.2 in \cite{Kerva1960}.
\end{proof}

By observing that $\Phi$ is additive with respect to the connected sum of manifolds, we have
\begin{corollary}\label{Phiequal}
For two $4$-connected closed $10$-dimensional $\mm{PL}$ manifolds $M$ and $M^\prime$, if $M$ is $\mm{BPL}\langle 8\rangle$-bordant to $M^\prime$, $\Phi(M)=\Phi(M^\prime)$. 
\end{corollary}

\begin{lemma}\label{Phibor}
	A $4$-connected closed manifold $M$ of dimension $10$ represents a nontrivial bordism class in $\Omega_{10}^{\mm{PL}\langle 8\rangle}$ iff $\Phi(M)=1$.
\end{lemma}
\begin{proof}
	Note that the PL sphere $S^{10}$ represents a trivial bordism class, $\Phi(S^{10})=0$. From \cite{Kerva1960}, there exists a closed 4-connected manifold $M_0$ of dimension $10$ for which $\Phi(M_0)=1$. Thus the theorem follows by Corollary \ref{Phiequal} and $\Omega_{10}^{\mm{PL}\langle 8\rangle}=\mb{Z}_2$.
\end{proof}

By Corollary \ref{bor-home} and Lemma \ref{Phibor}, we have 
\begin{theorem}
	The $B_5(M)$ and $\Phi(M)$ are a complete set of invariants of the $\mm{PL}$-homeomorphism type of the  $4$-connected closed manifold $M$ of dimension $10$.
\end{theorem}


\section{Classification for certain 13-manifolds}

In this section, let $\mc{M}$ be a PL manifold of dimension $13$ so that $H^\ast(\mc{M};\mb{Z})\cong H^\ast(\mm{CP}^3\times S^7)$.
As in \cite{ShenW}, we define the normal $B$-structure for the manifolds with {\it restriction lift } (see Definition \ref{reslift}).

By Remark \ref{PonStieclass}, denote the first Pontryagin class of $\mathcal M$ by $p_1(\mathcal{M}) =sx^2$ $s\in \mb{Z}$, 
the second Stiefel-Whitney class by  $w_2(\mathcal{M})=s x\mod{2}$. Let $\mathcal{H}$ denote the Hopf (vector) bundle over $\mathrm{CP}^4$. If $s\ge 0$, let $\xi$ be the complementary bundle of the Whitney sum $s\mathcal{H}$; if $s<0$, let $\xi$ be the Whitney sum $-s\mathcal{H}$. Formally $\xi=-s\mc{H}$.

Let   
$\pi: {B}=\mm{BPL}\langle 8 \rangle \times  \mm{CP}^4\to \mm{BPL}$ 
be the classifying map of the bundle $\gamma_8^{pl}\times \xi$ where $\gamma^{pl}_8$ is the universal bundle over $\mm{BPL}\langle 8 \rangle$, $\xi$ as a PL bundle forgets its vector bundle structure. Obviously, a manifold $\mc{M}$ with {\it restriction lift } has a normal $B$-structure i.e. its Gauss map has a lift to $B$.

All $n$-dimensional manifolds with a normal ${B}$-structure form a bordism group $\Omega_n^{\mm{PL}\langle 8 \rangle}(\xi)$ which is isomorphic to the homotopy group $\pi_n(\mm{M{B}})$ by the Pontryagin-Thom isomorphism. Here, $\mm{M{B}}=\mm{MPL}\langle 8 \rangle\wedge \mm{M}\xi $;  $\mm{M}\xi$ and $\mm{MPL}\langle 8\rangle$ are the Thom spectra of the bundles $\xi$ and $\gamma_8$. Note that if $p_1(\mc{M})=p_1(\mc{M}^\prime)$, $\mc{M}$ and $\mc{M}^\prime$ will be located in the same bordism group $\Omega_{13}^{\mm{PL}\langle 8 \rangle}(\xi)$.

By the same arguments as \cite{ShenW} and basic results of \cite{KirbSieb}, we have
 \begin{lemma}\label{hcobordism}
 Let $p_1(\mc{M})=p_1(\mc{M}^\prime)$. If $\mathcal M$ is $B-$coborant to $\mathcal M^\prime$, then there is a $\mm{PL}$ $h$-cobordism
  $(W,\mathcal M,\mathcal M^\prime).$
  \end{lemma} 

Now, consider the computation of $\Omega_{13}^{\mm{PL}\langle 8 \rangle}(\xi)$.
There is an Atiyah-Hirzebruch spectral sequence (AHSS)  for $\pi_\ast(\mm{MB})$
 $$E_2^{p,q}\cong H_p(\mm{M}\xi
 ;\pi_q(\mm{MPL}\langle 8\rangle))\Longrightarrow \mm{MPL}\langle 8\rangle_\ast(\mm{M}\xi)\cong \pi_\ast(\mm{MPL}\langle 8\rangle\wedge \mm{M}\xi).$$
By Proposition \ref{piMPL8}, $E_2^{p,13-p}=0$ for $0\le p\le 13$. Hence $\Omega_{13}^{\mm{PL}\langle 8 \rangle}(\xi)=0$.

The (1) of Theorem \ref{classify13} follows from Lemma \ref{hcobordism}. Thus, any $\mc{M}$ with {\it restriction lift} can be regarded as the total space of the sphere bundle of a $8$-dimensional vector bundle over $\mm{CP}^3$ (see Remark \ref{bundle}). The homotopy classification of the manifolds transforms to the fibre homotopy classification of $S^7$-bundles over $\mm{CP}^3$. Subsequently, the (2) of Theorem \ref{classify13} follows from the lemma 5.2 of \cite{ShenW}.

\end{document}